\documentclass{amsart}
\usepackage{graphicx}
\usepackage{amsthm}
\usepackage{amsmath}
\usepackage[utf8]{inputenc}
\usepackage{amsfonts} 
\usepackage[english]{babel}
\usepackage{tikz}
\usepackage{pgfplots}
\usepackage{mathrsfs}
\usepackage{hyperref}
\usetikzlibrary{arrows}
\newtheorem{Theorem}{Theorem}[section]
\newtheorem*{Theorem*}{Theorem}
\newtheorem{Proposition}{Proposition}[section]
\newtheorem{Lemma}{Lemma}[section]
\newtheorem{Definition}{Definition}[section]
  
\newtheorem{Remark}{Remark}[section]

\newtheorem{Corollary}{Corollary}[section]

\usepackage{xcolor}

\title{Graph skeletons and diminishing minors}

\author{Michael Bruner}
\address{University of Montana, 32 Campus Drive
Missoula, MT 59812}
\email{michael@leebruner.com}

\author{Atish Mitra}
\address{Montana Technological University, 1300 West Park Street Butte, MT 59701}
\email{amitra@mtech.edu}

\author{Heidi Steiger}
\address{Montana Technological University, 1300 West Park Street Butte, MT 59701}
\email{hsteiger@mtech.edu}


\begin{document}

\maketitle

\begin{abstract}
   We introduce the notion of coarse bottlenecking in graphs and coarse skeletons of graphs and show how bottlenecking guarantees that a skeleton resembles (up to quasi-isometry) the original graph. We show how these tools can be used to simplify the structure of graphs upto quasi-isometry that have an excluded asymptotic minor, reducing it to a skeleton of the original containing no $3$-fat minor. We give an example to show that a similar result does not hold for $2$-fat minors. This makes progress towards a Conjecture posed by Georgakopoulos and Papasoglu (\cite{GP23} Conjecture 1.1).
\end{abstract}

\section{Introduction}

Coarse geometry has natural connections with graph theory, several recent works have found new results and suggest some far reaching directions. In \cite{GP23} Georgakopoulos and Papasoglu list several such results by various groups of researchers, and herald the emergence of the field of ``Coarse Graph Theory".

Some of the themes of \cite{GP23} are looking at graph minors, their coarse versions, and the role they play in exploring the coarse geometry of graphs. The paper also looks at metric trees and stars, and, in combination with a paper by Fujiwara and Papasoglu \cite{FP23}, outerplanar graphs, to give coarse characterizations in terms of asymptotic minor exclusions and some metric properties.

Our work is inspired by some of the questions and techniques in the papers mentioned above. Section \ref{Section:CoarseSkel} gives a general approach to constructing a ``Coarse Skeleton" of a graph in terms of a ``Scale" and a ``Connectivity". We consider coarse bottlenecking (Definitions \ref{nEdgeB} \& \ref{mFnB}) as an natural theme to study for understanding the coarse geometry of graphs. In the case of $n=1,2$, coarse bottlenecking has been previously explored in the context of trees \cite{M05} and cacti \cite{FP23}. It has proved to be an interesting concept. We define $n$-edge bottlenecking and its coarse version and show that coarse bottlenecking is one method to ensure that the skeletons of a graph capture its structure. Towards unraveling connections between asymptotic and ordinary minor exclusion, we provide the following summary of main results.
\begin{Theorem*}\textbf{(See Theorem \ref{Theorem:StarvingMinor} and Theorem \ref{M,M reduces to 2})}
 
    Let $G$ be a coarsely bottlenecked graph that has some excluded asymptotic minor $H$.
    \begin{itemize}
        \item If $G$ contains an at most $M$ fat $H$-minor for $M>4$ then a $(2,2)$-skeleton of $G$ will contain an at most $(M+\frac{M}{2}-2)$-fat $H$ minor.
        \item There exists a graph $G'$ that is quasi-isometric to $G$ such that $G'$ contains at most a $2$-fat $H$ minor. Furthermore if $G$ contains no $M$-fat $H$ minor then such a $G'$ may be constructed by taking the $(M,M)$-skeleton of $G$.
    \end{itemize} 
\end{Theorem*}

Shortly after a prior version of this work was put on arXiv we were informed by Victor Chepoi of a related work \cite{CDNRV2010} that introduces and uses a notion of ``relaxed metric minors" that is similar to a fat minor and uses a construction called a ``layered partition".

James Davies and Victor Chepoi informed us of \cite{DHIM}, that both gives a counter example to the general case of \cite{GP23} Conjecture 1.1 and presents a different approach to quasi-isometrically producing graphs with at most $2$-fat $H$ minors from graphs with no asymptotic $H$ minor. In \cite{DHIM} it is shown how by adding a large number of edges to a graph that does not contain an asymptotic $H$ minor will produce a "power graph" that contains at most a $2$-fat $H$ minor. This is substantially different from the techniques presented in our Section \ref{Section:CoarseSkel} in several ways. The power graph approach has the advantage of giving a quasi-isometry for any graph, whereas the skeleton approach requires an additional assumption to ensure quasi-isometry. An advantage of our skeleton approach is that it produces a simpler graph, whereas the power graph is produced by adding many edges and will thus always contain the original graph as a subgraph.

\section{Preliminaries} \label{Prelims}
The coarse point of view was popularized by the groundbreaking work of Gromov \cite{G93}. We provide some terminology of graph theory and coarse geometry, we refer the reader to standard texts \cite{BMGT} and \cite{Roe} on the subjects for more details. Unless otherwise noted we will always assume that any graph or metric space we work with is connected and unbounded.
\subsection{Metric Geometry}

\begin{Definition}\textbf{(Length/Geodesic Spaces)}
    A metric space $X$ is said to be a length space if for any two points $x$ and $y$ in $X$, $d(x,y)$ is the infimum of the lengths of all $x,y$ paths. In cases where this infimum can always be achieved the space is said to be a geodesic space.
\end{Definition}

\begin{Definition}\textbf{(Graph, Path, Cycle, Graph Metric, Multi-edge, Tree)}
    Most of this work will be focused on the study of connected metric graphs. A graph $G$ is defined to be two sets, the vertices $V(G)$ and the edges $E(G)$. These sets have an inclusion relation between them such that each edge is associated with exactly two vertices, call these its endpoints. A path in a graph is a sequence of edges that joins a sequence of distinct vertices. A set $X\subset V(G)$ is said to be connected if there is a path between any two vertices in $X$ that contains only vertices from $X$. For $x,y\in V(G)$ an $x,y$ path is a path where the first vertex is $x$ and the last vertex is $y$. For sets $X,Y\subset V(G)$ an $X,Y$ path is an $x,y$ path for some $x\in X$ and some $y\in Y$. Two paths are said to be edge independent if they share no edges. Two paths are said to be internally disjoint if they share at most the first and last vertices, said to be the endpoints of the paths. A cycle is a pair of internally disjoint paths that share both their endpoints. There is a natural metric on the vertices of a graph, where the distances between two vertices $d(x,y)$, is taken to be the fewest number of edges in an $x,y$ path. We allow different edges of a graph to share both endpoints, this is called a multi-edge. If any two vertices have exactly one path between them a graph is said to be a tree.
\end{Definition}

\begin{Definition}\textbf{(Coarse/Quasi-Isometric Embedding and Equivalences)}
    For metric spaces $X,X'$ function $f:X\to X'$  is said to be a coarse embedding if there exist functions $\rho^+$ and $\rho^-$ such that for all $x,y\in X$.
    $$\rho^-(d(x,y))\leq d(f(x),f(y))\leq\rho^+(d(x,y))$$
    Furthermore $f$ is said to be a quasi-isometric embedding if there exist real constants $a\geq 1$ and $b\geq 0$ such that for all $x,y\in X$.
    $$\frac{1}{a}d(x,y)-b\leq d(f(x),f(y))\leq ad(x,y)+b$$\
    If there is a coarse embedding of $X\to X'$ and $X'\to X$ these spaces are said to be coarsely equivalent and, if there are quasi-isometric embedding with constants $a$ and $b$, then the spaces are said to be $(a,b)$-quasi-isometric and $f$ is a quasi-isometry. 
\end{Definition}
\begin{Remark}
    The composition of two coarse embeddings is a coarse embedding, the composition of two quasi-isometries is a quasi-isometry.
\end{Remark}

\begin{Remark}
    Any length space is quasi-isometric to some graph. As such, most results proven about the coarse geometry of a graph will extend to a length space.
    \end{Remark}
    \begin{proof}
        By taking any length space $X$, 
and fixing a constant $\epsilon>1$ we take a set of points $S$ such that the $\epsilon$ balls around $S$ cover $X$ and any two points in $S$ are more than $\epsilon$ apart. A graph can be constructed whose vertices are the points in $S$, with edges between them if $2\epsilon$ balls intersect. The natural function taking each $x\in X$ to the closest point of $S$ will be a quasi-isometry.
    \end{proof}

\begin{Definition}\textbf{($n$-edge bottlenecking)}\label{nEdgeB}
    A graph $G$ is said to have $n$-edge bottlenecking if for any two disjoint connected sets $X,Y\subset V(G)$, there exists a set $S\subseteq E(G)$ of size $n$ such that for any $X,Y$ path $P$ contains an element of $S$.
\end{Definition}

\begin{Definition}\textbf{(Minor)}
    A minor of a graph $G$ is a graph $H$ that can be produced from a subgraph of $G$ by a series of edge contractions, where a contraction of an edge is to remove the edge and merge its endpoints. A family of graphs is said to be minor-closed if every minor of a graph in the family is also in the family.
\end{Definition}
\begin{Remark}
    In \cite{GP23} an equivalent definitions for graph minors is given. We may borrow the terminology of \cite{GP23} where an $H$-minor is connected sets called the ``branch sets" corresponding to the vertices of $H$ and ``branch paths" corresponding to the edges of $H$, such that replacing each branch set with a vertex, and each branch path with an edge produces the graph $H$. Each of these definitions are coarsened in a similar way as in Section \ref{Section:CGT}.
\end{Remark}

\subsection{Coarse Graph Theory}
\label{Section:CGT}
\begin{Definition}\textbf{($M$-connected sets)}
    A set $X$ is said to be $M$-connected if for any two $x,y\in X$, there exists a sequence of points $x_0,x_1,x_2...x_n\in X$ such that $x=x_0,y=x_n$ and $d(x_i,x_j)\leq M$ if $|i-j|\leq 1$.
\end{Definition}

\begin{Remark}
    For sets of vertices in a graph being $1$-connected is equivalent to being connected.
\end{Remark}

\begin{Definition}\textbf{($M$-disjoint sets)}
    Two sets $X,Y$ are said to be $M$-disjoint if for all $x\in X$ and all $y\in Y$ $d(x,y)>M$.
\end{Definition}

\begin{Definition}\textbf{($M$-neighborhood/Coarse Image of a Set)}
    The $M$ neighborhood of a set $X$ denoted as $N_M(X)$ contains all $y\in G$ where there exists an $x\in X$ such that $d(f(x),y)<M$. That is $N_M(X):=\{y\in G | \text{ there exists } x\in X,  d(f(x),y)<M\}$. Let $G$ and $G'$ be graphs, let $f:G\to G'$ be a quasi-isometry and $M\in\mathbb{R}^+$. The $M$-coarse image of some $X\subseteq G$, denoted $N_M(f(X))$ is the set of elements of $G'$ that are within $M$ distances of an element of the image of $X$. That is $N_M(f(X)):=\{y\in G' | \text{ there exists } x\in X,  d(f(x),y)<M\}$.
\end{Definition}

\begin{Definition}\textbf{($M$-fat/Coarse $n$-bottlenecking)}\label{mFnB}
    A graph $G$ is said to have $M$-fat $n$-bottlenecking for some $M,n\in\mathbb{N}$ if for any two connected, $M$-disjoint subgraphs $X,Y\subset G$, there exists a set $S\subset V(G)\setminus (V(X)\cup V(Y))$ of size $n$ such that every $X,Y$ path intersects $N_M(S)$. If a graph is $M$-fat $n$-bottlenecked for some $M$ it is said to be coarsely $n$-bottlenecked.
        \begin{Remark}
            $M$-fat $n$-bottlenecking implies $M'$-fat $n'$-bottlenecking for all $M'\geq M$ and $n'\geq n$. This shows that coarse bottlenecking defines a spectrum of graphs, see \cite{BMS} for an exploration of several properties related to this spectrum.
        \end{Remark}
\end{Definition}

\begin{Definition}\textbf{(M-fat/Asymptotic Minor)} (See also \cite{GP23})
    A graph $H$ is said to be a $M$-fat minor of $G$ if $H$ is a minor of $G$ that can be produced in such a way that the sets that contract to form the edges and vertices of $H$ are $M$-disjoint unless they are incident in $H$. If $G$ contains a $M$-fat $H$ minor for all $M\in\mathbb{N}$, $H$ is said to be an asymptotic minor of $G$.
\end{Definition}

The next 2 lemmas collects several general results that establish some techniques we may use in coarse graph theory. Lemma \ref{Lemma:IsAQI} gives several simple operations that may be combined together to quasi-isometrically manipulate a graph, while Lemma \ref{Lemma:LongLemma} outlines several coarse properties of a graph that are preserved under quasi-isometry.

\begin{Lemma}\label{Lemma:IsAQI}
    Given a graph $G$, a quasi-isometric graph $G'$ may be produced by any of the following operations. In each case $G'$ is produced from $G$ and there is a natural map from the vertices of $G$ to the vertices of $G'$.
    \begin{enumerate}
    \item Adding an single edge between two vertices. 
        \begin{proof}
        Consider $G'$ to be $G$ with an edge added between two vertices $x$ and $y$. Now for any vertices $x',y'$ any path in $G$ is a path in $G'$ so $d_{G'}(x',y')\leq d_G(x',y')$. Furthermore any path in $G'$ that does not exist in $G$ must use the newly added edge, and thus could instead use a path of length $d_G(x,y)$ thus $d_G(x',y')\leq d_{G'}(x',y')+d_G(x,y)$
        \end{proof}
    \item Adding many edges between pairs of vertices with distance bounded by some $M\in \mathbb{N}$. 
        \begin{proof}
            Consider $G'$ to be $G$ with many (possibly infinitely many) edges added where the endpoints of an added edge have distance in $G$ at most $M$. Now for any vertices $x',y'$ as any path in $G$ is a path in $G'$ $d_{G'}(x',y')\leq d_G(x',y')$. As any path in $G'$ may have every edge that is not in $G$ replaced by a path of at most $M$ we have $d_G(x',y')\leq Md_G'(x',y')$.
        \end{proof}
    \item Removing an edge between vertices that lie on a common cycle.
        \begin{proof}
        Consider $G'$ to be $G$ with some edge between two vertices $x$ and $y$ that lie on a common cycle removed. Now for any vertices $x',y'$ as every path in $G'$ exists in $G$ we have $d_G(x',y')\leq d_{G'}(x',y')$. As any path in $G$ that does not exist in $G'$ must use the removed edge, and thus can use the other $x,y$ path around the cycle we have $d_G(x',y')\leq d_{G'}(x',y')+d_{G'}(x,y)$.
        \end{proof}
        
        A second proof of (3) would be to observe that $G$ is $G'$ with an edge added and thus we have a quasi-isometry by (1). A similar argument gives the following Corollary of (2).
        
    \begin{Corollary}
        Removing many edges between many pairs of vertices so long as there remains a path of length bounded by some $M\in \mathbb{N}$ between the endpoints of any removed edge.
    \end{Corollary}
    \item Subdividing every edge $M\in\mathbb{N}$ times.
    \begin{proof}
        Consider $G'$ to be $G$ with every edge subdivided $M$ times. Now any path in $G$ has each edge replaced by a path of length $M$, so $d_G(x,y)\leq d_{G'}(x,y)\leq Md_G(x,y)$
    \end{proof}
    \begin{Corollary}
        Subdividing each edge in a set of edges at most $M\in \mathbb{N}$ times.
    \end{Corollary}
    \item Contracting a single edge.
    \begin{proof}
        Consider $G'$ to be created from $G$ by merging two vertices that have an edge between them in $G$. Any path in $G$ that used this edge has had its length reduced by $1$ and any path in $G'$ that contains this vertex gives at most a path one longer in $G$ so we have $d_{G'}(x,y)\leq d_G(x,y)\leq d_{G'}(x,y)+1$
    \end{proof}
    \begin{Corollary}\label{Corollary: uniformly bounded}
        Contracting a set where the connected components are uniformly bounded.
    \end{Corollary}
    \begin{Remark}
        One must be cautious to ensure that the connected components are bounded, as an example consider an infinite ray, we may contract any edge or set of finitely many edges, but we may not contract every edge as the union of all edges is connected and unbounded.
    \end{Remark}
\end{enumerate}
\end{Lemma} 

\begin{Lemma}\label{Lemma:LongLemma}
   Let $G$ and $G'$ be graphs, let $f:G\to G'$ be an $(a,b)$-quasi-isometry.
\begin{enumerate}
    \item For $M>a+b$ any connected $H\subset G$, $f(H)$ is $M$-connected set.
    \begin{proof}
        For any two $f(x),f(y)$ in $f(H)$ there are vertices $x,y$ in $H$ with a path consisting of vertices $p_0,p_1,p_2...,p_N$ such that $p_0=x,d(p_n,p_n+1)=1,p_N=y$. The images of these vertices make a set $f(p_0),f(p_1),f(p_2)...,f(p_N)$ such that $f(p_0)=f(x),d(f(p_n),f(p_n+1))\leq a+b,f(p_N)=f(y)$. This gives that the set $f(H)$ is $M$-connected.
    \end{proof}
    \begin{Corollary}\label{Corallary:(a+b)connected}
        The $(a+b)$-coarse image of a connected set is connected.
    \end{Corollary}
    \item For $M>a+ab$, any two vertices $x,y$ that are $M$-disjoint $f(x)\neq f(y)$.
    \begin{proof}
         $d(f(x),f(y))\geq\frac{1}{a}d(x,y)-b$ and so if $d(x,y)>a+ab$ we have  $d(f(x),f(y))\geq\frac{a+ab}{a}-b=1$ and thus $f(x)\neq f(y)$.
           
    \end{proof}
     \begin{Corollary}\label{Corallary:N(a+ab)-disjoint}
        For any two vertices $x,y$ that are $N(a+ab)$-disjoint, $d(f(x),f(y))\geq N$. 
    \end{Corollary}
    \item  If $G$ is a graph that is not $M$-fat $1$-bottlenecked for $M>2(a+b)(a+ab)$ then $G'$ is not $1$-edge bottlenecked.

    \begin{proof}
        Let $G$ be a graph that is not $2(a+b)(a+ab)$-fat $1$-bottlenecked. This gives that there are connected sets $X,Y\subset G$ with $d(X,Y)>2(a+b)(a+ab)$ such that there are at least two $2(a+b)(a+ab)$-disjoint paths $P_1,P_2$ between $X$ and $Y$. Now corollary \ref{Corallary:N(a+ab)-disjoint} gives that $$N_{a+b}(f(X))\cap N_{a+b}(f(Y))=\emptyset =N_{a+b}(f(P_1))\cap N_{a+b}(f(P_2)).$$
            While Corollary \ref{Corallary:(a+b)connected} gives that $N_{a+b}(f(P_1))$ and $N_{a+b}(f(P_2))$ both contain a path between $N_{a+b}(f(X))$ and $N_{a+b}(f(Y))$.
        This implies that $G'$ is not $1$-edge bottlenecked.
    \end{proof}    
    \begin{Corollary}\label{Q_BN=Fat_BN}
        If a graph $G$ is quasi $n$-edge bottlenecked then $G$ is $m$-fat $n$-bottlenecked for some $m$. 
    \end{Corollary}
\end{enumerate}
\end{Lemma}

\section{Coarse skeletons of a graph}\label{Section:CoarseSkel}

\begin{Definition}\textbf{(Coarse Skeleton, Layers, Blocks, $G_{\lambda,k}$)}
\label{definition: skeleton}
    For any connected graph $G$, we define a ``Skeleton" construction $G_{\lambda,k}$ to capture large-scale properties of $G$. To form this construction, we first pick a vertex $x_0$ in $G$ to act as a ``root" and fix two constants, a ``scale" $\lambda\geq 1$, and ``connectivity" $k\geq 1$. We will often pick one constant as a function of the other. The scale divides the graph into annuli of radius $\lambda$ around the root. We call these the ``layers" where the $N$th layer is the set $A_{N,\lambda} := \{x\in V(G) \mid N\lambda<d(x,x_0)\leq (N+1)\lambda\}$. The connectivity partitions each layer into maximal $k$-connected sets. We call these the ``blocks". We now define $G_{\lambda,k}$ to have a vertex identified with each block, and an edge between two vertices of $G_{\lambda,k}$ if there is an edge between the blocks in $G$. We will often assume a natural function $f:G\to G_{\lambda,k}$ defined by taking each vertex of $G$ to the vertex of $G_{\lambda,k}$ corresponding to its block. See Figure \ref{An sets} for a visual representation of one such construction.
\end{Definition}

\begin{Remark}\label{remark: skeleton facts}
    There are several facts that can be observed about a skeleton of a graph.
    \begin{enumerate}
        \item Skeleton $G_{\lambda,k}$ is a bipartite graph, as every edge connects from one layer to a consecutive one.
        \item Skeleton $G_{\lambda,k}$ is a connected graph, as every vertex has a path to the root.
        \item Skeleton $G_{\lambda,k}$ is a simple graph, as we only create at most one edge between two blocks.
        \item All edges in $G_{\lambda,k}$ come from at least one edge in G.
        \item All vertices in $G_{\lambda,k}$ come from at least one vertex in G.
    \end{enumerate}
\end{Remark}

\begin{figure}
    \centering
    \includegraphics[width=0.75\linewidth]{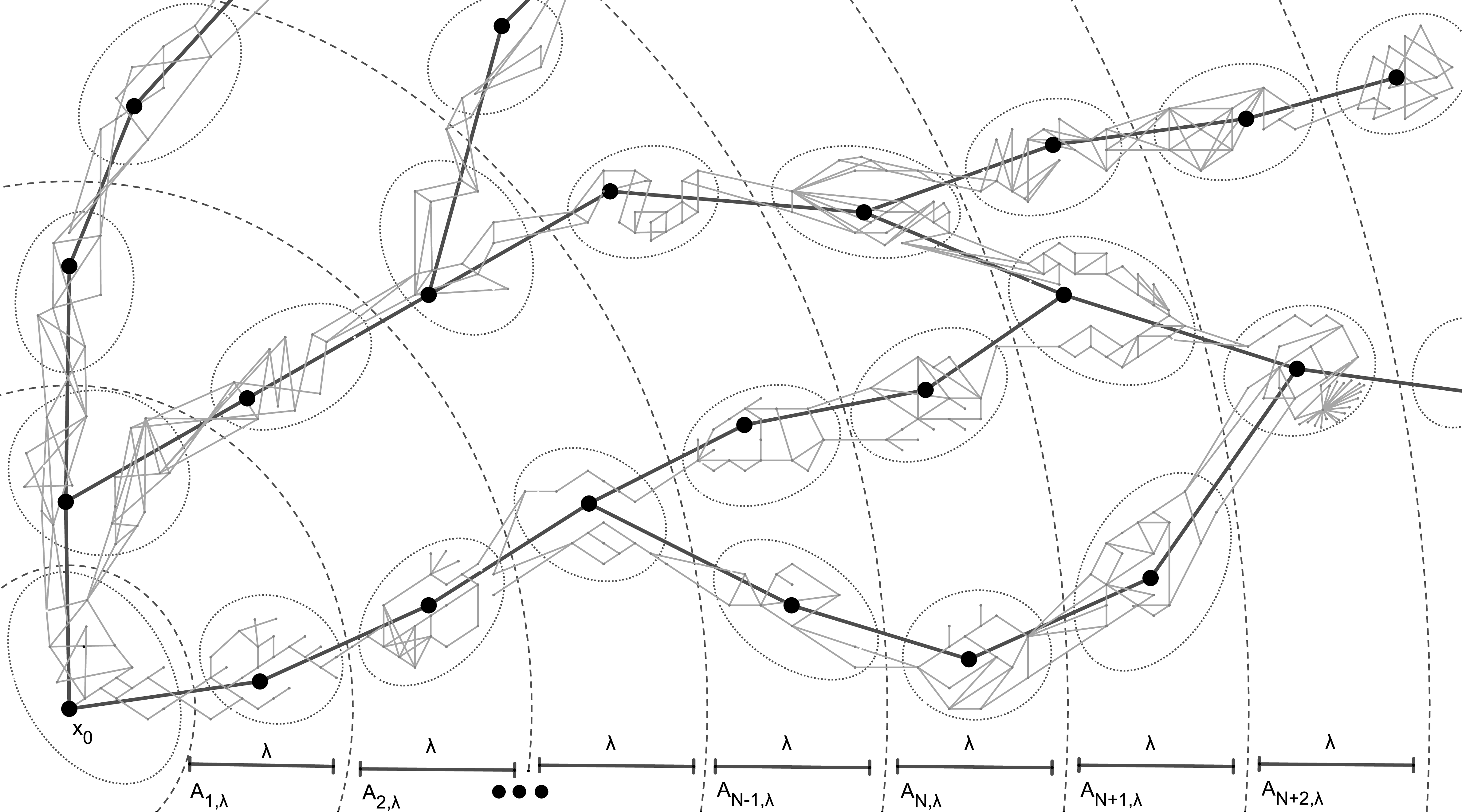}
    \caption{(See definition \ref{definition: skeleton})
    Graphical depiction of a skeleton construction from a graph $G$ (in grey). layers $\lambda$ from the root $x_0$ are shown as dashed arcs and k-connected block are surrounded by ovals. vertices and edges of $G_{N,\lambda}$ ares shown as black with one vertexes associated with each block and an edge between blocks if there is a corresponding edge in $G$. Note that blocks are not necessarily connected. }
    \label{An sets}
\end{figure}

\begin{Lemma}\label{uniformly bounded = quasi-isometry}
    The natural function $f:G\to G_{\lambda,k}$ is a quasi-isometry iff the blocks are uniformly bounded.
\end{Lemma}
\begin{proof}
    Let $G$ be a graph and $G_{\lambda,k}$ be constructed as defined. Assume that the blocks are uniformly bounded by some $M\in\mathbb{N}$. For any $x,y$ vertices in $G$, we have $d(f(x),f(y))\leq d(x,y)$, as every edge in $G_{\lambda,k}$ comes from an edge in $G$. There is a path in $G_{\lambda,k}$ of length $d(f(x),f(y))$. Each edge in $G_{\lambda,k}$ comes from an edge in $G$ and consecutive edges in $G_{\lambda,k}$ coming from two edges in $G$ with one endpoint inside of the same block bounded by $M$. Thus, we have $d(x,y)\leq M d(f(x),f(y))+2M$ meaning $f$ is a quasi-isometry. 
\end{proof}

\begin{Theorem}
    If a graph $G$ is $M$-fat $n$-bottlenecked for some $M$ and $n$. Then $f:G\to G_{M,M}$ is a quasi-isometry.
    \label{if BN then QI}
\end{Theorem}
  
\begin{Corollary}
    A coarsely bottlenecked graph has asymptotic dimension 1 (see \cite{BD08} for information on asymptotic dimension).
    The converse is not true.
\end{Corollary}
\begin{proof}(See Figure \ref{bounded_sets})
    Let $G$ be a $M$-fat $n$-bottlenecked graph for some fixed $M,n\in\mathbb{N}$. By Lemma \ref{uniformly bounded = quasi-isometry}, if all the blocks are uniformly bounded then $f$ is a quasi-isometry, so assume to a contradiction that for any bound there is some block with diameter greater than it. Let $B$ be one such block with diameter larger than $10M(n+1)$. As the diameter is greater than $10M(n+1)$ there are vertices $x_1,x_{n+1}\in B$ such that $d(x_1,x_{n+1})>10M(n+1)$. As $B$ is $M$-connected this gives vertices $x_2,x_3,x_4...x_n \in B$ such that all are more than $7M$-disjoint. Following  paths from each $x_i$ to the root for $3M$ gives $n+1$ $M$-disjoint paths between $B$ and a  $(d(B,x_0)-3M)$-neighborhood around the root. This gives two $M$-disjoint sets connected by $n+1$ $M$-disjoint paths that cannot all be intersected by the $M$-neighborhood around a set of a set of $n$ vertices, producing a contradiction. 
\end{proof}
\begin{Corollary}\label{QI_Contition}
    If a graph $G$ is coarsely bottlenecked, then $f:G\to G_{\lambda,k}$ is a quasi-isometry for all $\lambda$ and all $k$.
\end{Corollary}

\begin{figure}
    \begin{center}
\includegraphics[width=0.75\linewidth]{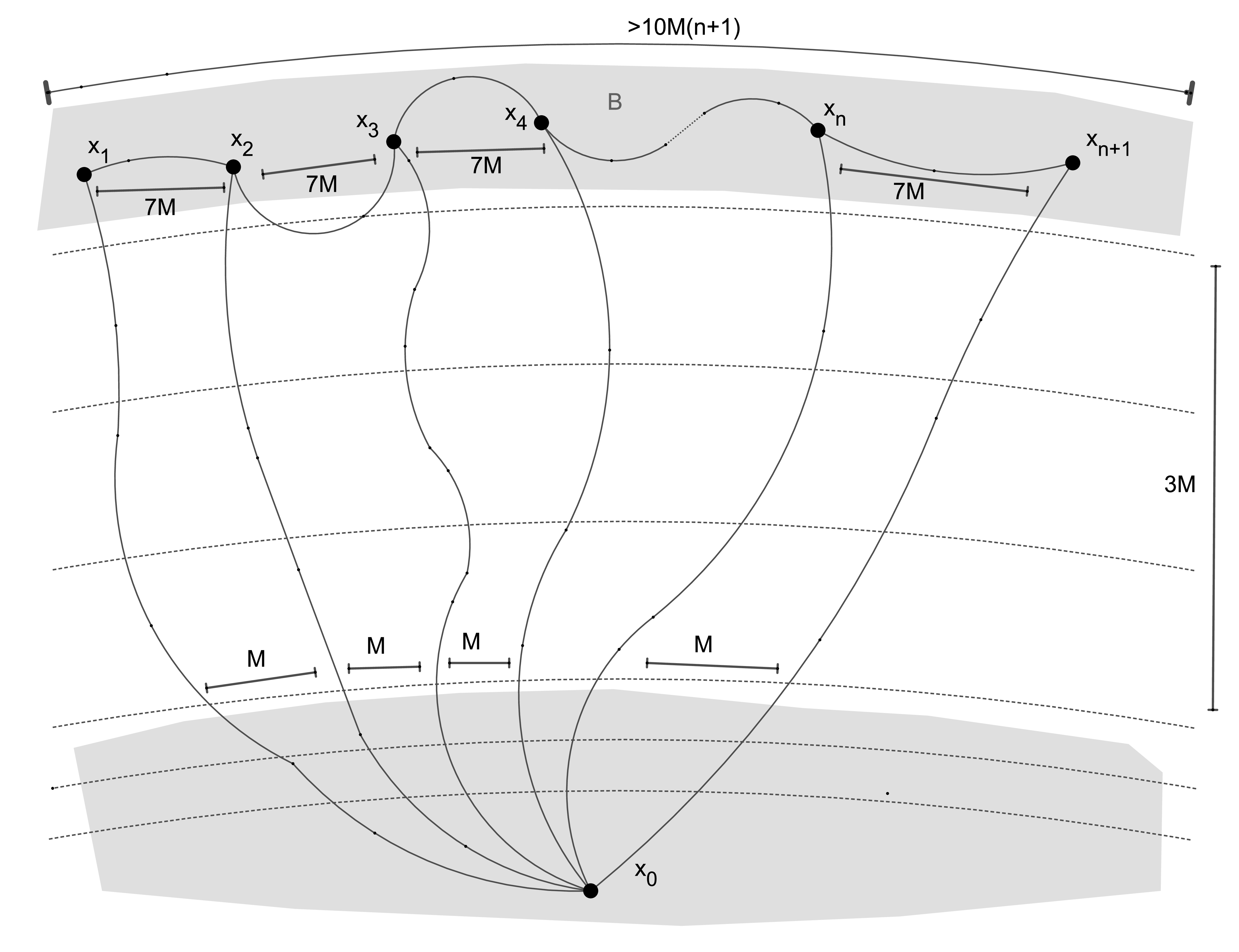}
    \end{center}
    \caption{ Picture proof of Theorem \ref{if BN then QI} showing that an unbounded block creates a not coarsely bottlenecked structure. Block B and the $(d(B,x_0)-3M)$-neighborhood around the root $x_0$ are the shaded areas. As $x_1,x_2,x_3,x_4...x_n, x_(n+1)$ are all pairwise $7M$-disjoint, their $3M$ paths to the root always remain at least $M$-disjoint.}
    \label{bounded_sets}
\end{figure}

\begin{Lemma}\label{Lemma:NoEdgeGivesDisjoint}
    If two distinct vertices in $G_{\lambda,k}$ are not connected by an edge then their preimages in $G$ are at least $M=\min\{\lambda,k\}$-disjoint. 
\end{Lemma}
\begin{proof}

Let $x,y$ be vertices of $G_{\lambda,k}$ such that $f^{-1}(\{x\})$ and $f^{-1}(\{y\})$ are not $M$-disjoint but with $x \neq y$, we will show that $x$ and $y$ must have an edge between them. As the preimages of $x$ and $y$ are not $\lambda$-disjoint, $x$ and $y$ must either be in the same layer or consecutive layers. As $x$ and $y$ are distinct vertices and their preimages are not $k$-disjoint, $x$ and $y$ can not be in the same layer and thus must be in consecutive layers. As $f^{-1}(\{x\})$ and $f^{-1}(\{y\})$ are not $M$-disjoint there are vertices $x'\in f^{-1}(\{x\})$ and $y'\in f^{-1}(\{y\})$ such that $d_G(x',y')<M$. Thus there must be a $x',y'$ path $P$ that is never $M$-disjoint from $x'$ and $y'$. This gives that $P$ must be contained only in the layers of $x'$ and $y'$, and that in these layers it is contained in $f^{-1}(\{x\})$ and $f^{-1}(\{y\})$. By following this path from $x'$ to the first vertex in $f^{-1}(\{y\})$ it becomes clear that there is an edge between $f^{-1}(\{x\})$ and $f^{-1}(\{y\})$ in $G$ and thus an edge between $x$ and $y$ in $G_{\lambda,k}$.
\end{proof}

\begin{Corollary}
\label{NoRungConnectionsGivesMdisjoint}
    Let $G_{\lambda,k}$ be a skeleton of a graph with paths $P_1,P_2$. If $d(f(x),f(y))>1$ for all $f(x)\in P_1,f(y)\in P_2$, then for $M=\min\{\lambda,k\}$ the half $M$ neighborhoods of the preimages do not intersect. That is $N_\frac{M}{2} f^{-1}(P_1)\cap N_\frac{M}{2} f^{-1}(P_2)=\emptyset$.
\end{Corollary}
\begin{Corollary}\label{Lemma:N disjoint gives M(N-1) disjoint}
    If two vertices $x,y\in G_{\lambda,k}$ with $d(x,y)\geq N$ must have preimages in $G$ that are at least $M(N-1)$-disjoint for $M=\min\{\lambda,k\}$.
\end{Corollary}

\subsection{Example use of Skeletons: Quasi Trees}
\begin{Remark}
    This equivalence was first shown by \cite{M05} using a condition that is equivalent to $1$-coarse bottlenecking, \cite{GP23} provided a simple proof using effectively the same techniques we demonstrate here. Our contribution here is to present this proof in the formalised language of skeletons. 
\end{Remark} 
\begin{Proposition}
    
(See also \cite{GP23},\cite{M05}) The following are equivalent:
\begin{enumerate}
    \item A graph $G$ is $M$-fat $1$-bottlenecked for some $M\in\mathbb{N}$.
    \item For $m>M$, $G_{m,m}$ is $1$-edge bottlenecked.
    \item $G$ is a quasi-tree, as it is quasi-isometrically equivalent to a tree graph. 
\end{enumerate}    
\begin{proof}{$1\implies 2$}
    Let $G$ be a $M$-fat $1$-bottlenecked for some $M\in\mathbb{N}$. Fix $m>M$ and assume to a contradiction that $G_{m,m}$ is not $1$-edge bottlenecked. This gives that there are $x\neq y \in  G_{m,m}$ such that there are two internally disjoint $x,y$ paths in $G_{m,m}$. Note as $G_{m,m}$ is bipartite, every cycle must have two vertices in the same layer. Thus we may consider an $x\neq y$ pair that lie in the same layer meaning their preimages are not $m$-connected.
    Now both of the $x,y$ paths gives at least one path in $G$. As $G$ is $m$-fat $1$-bottlenecked these paths are not $m$-disjoint in $G$. Thus there must be an edge between them in $G_{m,m}$. Taking $y'$ to be the closest endpoint of such an edge to $x$, gives a pair of disjoint $x,y'$ paths with no edges between them. Now the paths given in $G$ produce $m$-disjoint sets that connect $f^{-1}(\{x\}),f^{-1}(\{y'\})$. This is a contradiction and thus $G_{m,m}$ is $1$-edge bottlenecked.
    
{$2\implies 3$}

     Corollary \ref{QI_Contition} gives that the function $f:G\to G_{m,m}$ is a quasi-isometry. As $G_{m,m}$ is $1$-edge bottlenecked it must be a tree. Thus $G$ is a quasi-tree. 
     
{$3\implies 1$}

If a graph $G$ is a quasi-tree there exists a tree $G'$ and a quasi-isometry $f:G\to G'$ with constants $a,b$. Thus by Corollary \ref{Q_BN=Fat_BN} $G$ is $M$-fat $1$-bottlenecked for some $M$.
\end{proof}
\end{Proposition}

\begin{Remark}
    For more results relating to the properties of coarsely bottlenecked graphs please see \cite{BMS}. Where we explore bottlenecking in much more detail.
\end{Remark}

 \subsection{Composition of Skeletons, and Diminishing Minors}

\begin{Proposition}\label{Proposition:CompositionOfSkeletons}
Fixing $x_0$ and taking a $(N,1)$-skeleton of a $(\lambda,k)$-skeleton of a graph is equivalent to taking a $(N\lambda,k)$-skeleton of that graph 
    $G_{(N\lambda,k)}=G_{(\lambda,k)_{(N,1)}}$ for $N\in\mathbb{N}$.
\end{Proposition}

\begin{proof}
    Let $f:G\to G_{\lambda,k}, \phi:G\to G_{N\lambda,k}$ and $\Phi:G\to G_{(\lambda,k)_{(N,1)}}$ be the natural functions found by mapping a block in one to its associated vertex in the other. Let $x,y$ be vertices in $G$ we will show that $\phi(x)=\phi(y) \iff \Phi(x)=\Phi(y)$.
    
    If $\phi(x)=\phi(y)$ then $x$ and $y$ are both in the same $N\lambda$-layer, and are $k$-connected within the layer. This means that $f(x)$ and $f(y)$ are connected in an $N$-layer of $G_{(\lambda,k)}$ and so $\Phi(x)=\Phi(y)$.

    If $\Phi(x)=\Phi(y)$ then $f(x)$ and $f(y)$ are connected in a $N$-layer of $G_{(\lambda,k)}$ and so $x$ and $y$ are $k$-connected in some $N\lambda$-layer of $G$, so $\phi(x)=\phi(y)$.

    As the vertex sets of these graphs are the same, and all edges are determined by the edges between vertex sets, we have shown that these graphs are the same.
\end{proof}

\begin{Remark}
    The above Proposition \ref{Proposition:CompositionOfSkeletons} shows an interesting relation between the "scales" of skeleton compositions. Unfortunately a similar result does not hold for "connectivities".
\end{Remark}

\begin{Lemma}[Small Contraction Lemma ]\label{Lemma: SmallContraction}
    Vertices $2$-disjoint in $G_{(\lambda,k)_{(2,2)}}$ must come from sets that are $3$-disjoint in $G_{\lambda,k}$.
\end{Lemma}

\begin{proof}
    All shortest paths between vertices that are $2$-disjoint in $G_{(\lambda,k)_{(2,2)}}$ have at least two consecutive edges forming a path. These edges either cross the same layer or consecutive layers. Edges crossing the same layer place the endpoint vertices of the path in the same layer. Two distinct vertices of $G_{(\lambda,k)_{(2,2)}}$ in the same layer must come from sets that are not $2$-connected in $G_{\lambda,k}$, and must be $3$-disjoint. Consecutive edges crossing subsequent layers places a layer between the endpoints vertices of the path. Two vertices of $G_{(\lambda,k)_{(2,2)}}$ with a layer between them must come from sets that are not $2$-connected $G_{\lambda,k}$, so must be $3$-disjoint.
\end{proof}

\begin{Lemma}[Big Contraction Lemma]\label{Lemma: BigContraction}
    Let $G_{\lambda,k}$ be a skeleton of a graph, if two sets are $n>1$-disjoint in $G_{(\lambda,k)_{(2,2)}}$ then their pre-images in $G_{\lambda,k}$ are at least ($n+\lfloor\frac{n}{2}\rfloor)$-disjoint.
\end{Lemma}

\begin{proof}
      Each shortest path between two n-disjoint sets in $G_{(\lambda,k)_{(2,2)}}$, has at least $n+1$ vertices. Consider alternating vertices along this path and see that there are $\lfloor\frac{n}{2}\rfloor$ number of $2$-disjoint consecutive pairs. Applying lemma \ref{Lemma: SmallContraction} along these paths in $G_{(\lambda,k)_{(2,2)}}$ increases the minimum necessary lengths of the pre-image paths in $G_{\lambda,k}$ by at least $1$ for each such pair. This means $n$-disjoint sets in $G_{(\lambda,k)_{(2,2)}}$ must come from at least $(n+\lfloor\frac{n}{2}\rfloor)$-disjoint sets in $G_{\lambda,k}$. 
\end{proof}

\begin{Lemma}[Diminishing Lemma]\label{Lemma:Diet}
   for $M\geq 4$, If a coarsely bottlenecked graph G has skeleton $G_{\lambda,k}$ such that $G_{(\lambda,k)_{(2,2)}}$ contains a $M$-fat $H$ minor, then $G_{\lambda,k}$ contains a $(M+\lfloor\frac{M}{2}\rfloor-2)$-fat $H$ minor.
\end{Lemma}

\begin{proof}
    As $G_{(\lambda,k)_{(2,2)}}$ contains a $M$-fat $H$ minor for $M\geq 4$, the branch sets and paths are at least $M$-disjoint except where they are incident in $H$. By Lemma \ref{Lemma: BigContraction} the preimages of the sets are $(M+\lfloor\frac{M}{2}\rfloor)$-disjoint. These sets are $2$-connected and so the $1$-neighborhoods around them are are connected. This gives that the $1$-neighborhoods are at least $(M+\lfloor\frac{M}{2}\rfloor-2)$-disjoint except where incident in $H$. This gives a $(M+\lfloor\frac{M}{2}\rfloor-2)$-fat $H$ minor in $G_{\lambda,k}$.    
\end{proof}

\begin{Corollary}
\label{Corollary: (M+1)-fat}
     Let $G_{\lambda,k}$ be a skeleton of a graph such that $G_{(\lambda,k)_{(2,2)}}$ contains a $M$-fat $H$ minor for $M\geq 6$, $G_{\lambda,k}$ contains at least a $(M+1)$-fat $H$ minor.
\end{Corollary}

\begin{Theorem}[Diminishing Minor Theorem]\label{Theorem:StarvingMinor}
   If a coarsely bottlenecked graph $G$ does not contain a graph $H$ as an asymptotic minor then $G$ is quasi-isometric to a graph $G'$ with no $5$-fat $H$ minor. Furthermore such a $G'$ may be constructed by repeatedly taking $(2,2)$-skeletons of $G$. The constants of this quasi-isometry may be determined by the coarse bottlenecking number of $G$ as well as the scale where $G$ does not contain $H$ as a fat minor.
\end{Theorem}

\begin{proof}
    Let $G$ be a coarsely bottlenecked graph that does not contain $H$ as an asymptotic minor. As $G$ is coarsely bottlenecked, it is quasi-isometric to a skeleton $G_{2,2}$. As asymptotic minors are preserved under quasi-isometry, $G_{2,2}$ does not contain $H$ as an asymptotic minor. This gives that there is some $M$ such that $G_{2,2}$ contains at most an $M$-fat $H$ minor. Take $M+1$ consecutive $(2,2)$-skeletons of $G_{2,2}$ to produce a graph $G'$, note that $G$ is quasi-isometric to $G'$. Now assume to a contradiction that $G'$ has a $4$-fat $H$ minor, corollary \ref{Corollary: (M+1)-fat} of Lemma \ref{Lemma:Diet} gives that $G_{2,2}$ now must contain a $M+1$-fat $H$ minor, this is a contradiction and so $G'$ must contain at most a $4$-fat $H$ minor.
\end{proof}

    It is of note that Theorem \ref{Theorem:StarvingMinor} can only create a quasi-isometric graph without $4$-fat $H$-minors, the limitation in further reduction is comes from two restrictions. First the contraction properties of skeletons, where large distances are always reduced, but distances of one are preserved if the edge crosses between two blocks. Second the need to take $\frac{k}{2}$-balls around the branch sets and paths to ensure that they are still connected.
    Theorem \ref{Theorem:StarvingMinor} reduces non-asymptotic minors to at most $3$-fat, this can be improved slightly achieved by relying on a Corollary \ref{Lemma:N disjoint gives M(N-1) disjoint} of Lemma \ref{Lemma:NoEdgeGivesDisjoint}, as seen in Theorem \ref{M,M reduces to 2} which reduces minors to at most $2$-fat. Theorem \ref{Theorem:Hammer} shows that improving this any further will require more assumptions about $H$, or the use of additional techniques beyond just skeletons. Recently, the authors of \cite{DHIM} have shown, by example, that the bound from Theorem \ref{M,M reduces to 2} is optimal in the general case.

\begin{Theorem} \label{M,M reduces to 2}
    Let $G$ be a coarsely bottlenecked graph, and let $H$ be a graph such that $G$ does not contain $H$ as an asymptotic minor. For any $M\in\mathbb{N}$ such that $G$ does not contain a $M$-fat $H$ minor, $G_{M,M}$ contains no $3$-fat $H$-minor. 
\end{Theorem}
\begin{proof}
    Let $G$ be coarsely bottlenecked graph that does not contain a $M$-fat $H$ minor. Assume to a contradiction that $G_{M,M}$ has a $3$-fat $H$ minor, this gives branch paths and sets in $G_{M,M}$ that are at least $3$-disjoint except where incident in $H$. By Lemma \ref{Lemma:N disjoint gives M(N-1) disjoint}, the preimages of these sets are at least $2M$-disjoint, so their $\frac{M}{2}$ neighborhoods are connected and at least $M$-disjoint. This gives a $M$-fat $H$ minor in $G$, and this is a contradiction.
\end{proof}

\begin{Remark}
    As a result of theorem \ref{M,M reduces to 2} and theorem \ref{Theorem:StarvingMinor}, we observe that the following weakened version of \cite{GP23} Conjecture 1.1 is true. This also follows from a result of \cite{DHIM}.
\end{Remark}

\begin{Corollary}(A weaker version of Conjecture 1.1 of \cite{GP23}) Let $X$ be a length space, and let $H$ be a finite graph. Then $X$ has no $K$-fat $H$ minor for some $K\in\mathbb{N}$ if and only if $X$ is quasi-isometric to a length space with no $\epsilon$-fat $H$ minor for any $\epsilon>0$. Furthermore, the constants of this quasi-isometry depend only on $K$, $H$ and $\epsilon$.
\end{Corollary}

\begin{proof}
    First apply Theorem \ref{M,M reduces to 2} to produce a simple bipartite graph with no minors that are greater than $2$-fat, then change the metric on this graph to a weighted metric with each edge counting for $\frac{\epsilon}{2}$.
\end{proof}

\begin{Remark}
    Seeing how Theorem \ref{M,M reduces to 2} is able to take any non-asymptotic minor to at most $2$-fat, one might hope an appropriate application of skeletons would be able to remove non-asymptotic minors of simple graphs entirely. This is not always the case. Theorem \ref{Theorem:Hammer} shows how it can fail for $D_3$ (two vertices joined by three multi-edges).
\end{Remark}
\begin{Lemma}\label{Lemma:Leaf}(See Figure \ref{fig: hammer})
    For a fixed $\lambda,k$ and $x_0$ there exists a graph $G'$ with an at most $2$-fat $D_3$ minor such that $G'_{\lambda,k}$ contains a $D_3$ minor. 
\end{Lemma}
\begin{proof}
    Construct graph $G'$ by starting at a vertex $x_0$ and attaching a path of length $\lambda - ( \frac{1}{2}k-1)$. At the end of this path, add two more edges and connect their endpoints with two paths of length $> 2\lambda+2k$. This structure is now a $2$-fat $D_3$ minor. As seen in the figure, $G'_{\lambda,k}$ contains a $D_3$ minor.
\end{proof}
\begin{figure}
   \centering    \includegraphics[width=1\linewidth]{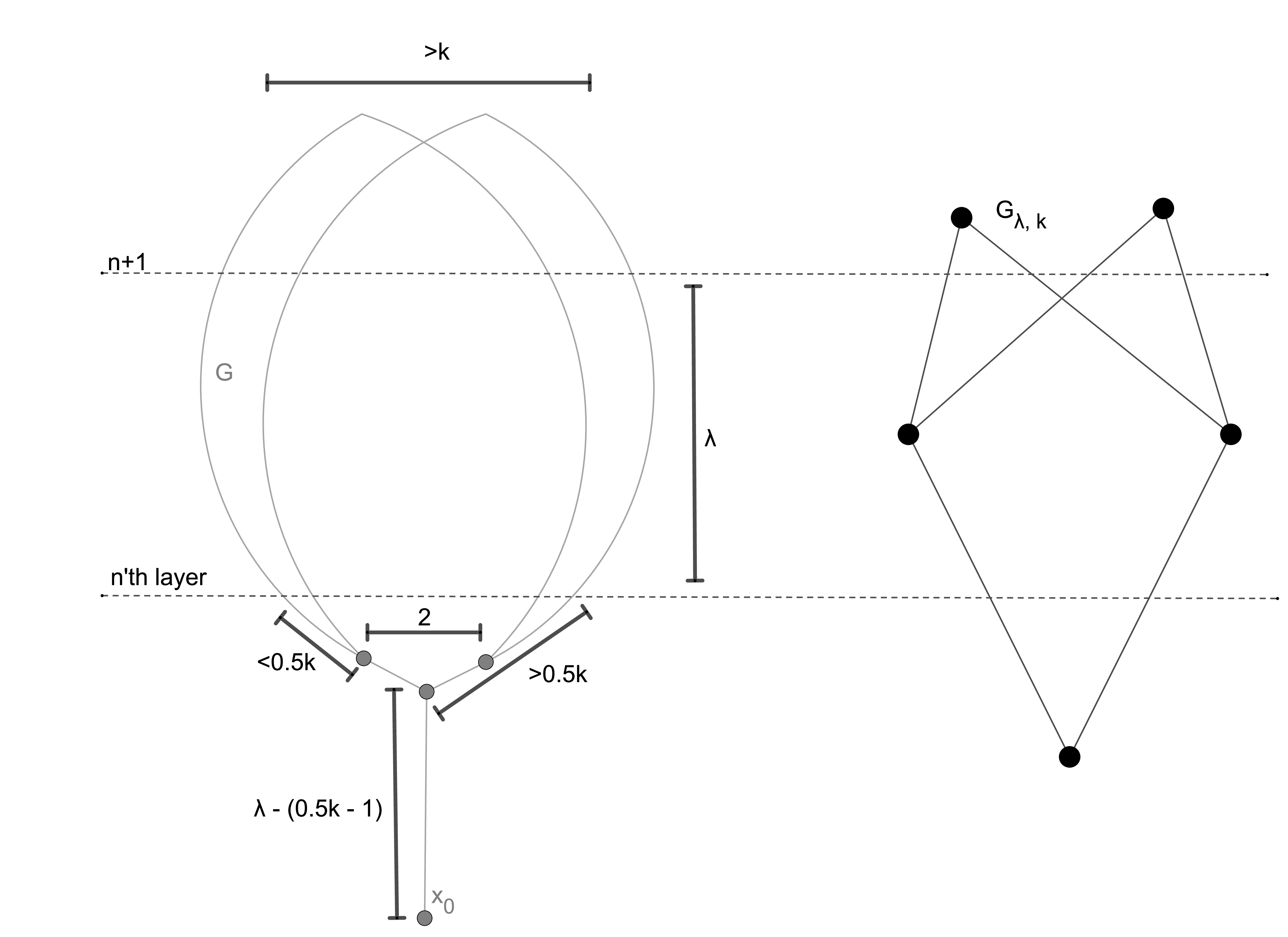}
    \caption{Illustration of lemma \ref{Lemma:Leaf} showing a non-asymptotic minor being preserved by any skeleton. Part of $G'$ is on the left in grey and its $G'_{\lambda,k}$ skeleton is in black on the right. layers are the dashed lines. Important distances are noted with measurement lines}
    \label{fig: hammer}
\end{figure}
\begin{Theorem}\label{Theorem:Hammer}
    There exists a graph $G$ with an at most $2$-fat $D_3$ minor such that every skeleton of $G$, for any choice of $\lambda,k$ and $x_0$ contains a $D_3$ minor.
\end{Theorem}
\begin{proof}
    Construct $G$ by taking an infinite ray, and at each vertex attaching a copy of $G'$ as constructed in Lemma \ref{Lemma:Leaf} for every possible combination of $\lambda$ and $k$. Now for any choice of $\lambda,k$ and $x_0$ the ray ensures that a copy of $G'$, for the appropriate $\lambda$ and $k$ is placed at a distance from $x_0$ such that $G_{\lambda,k}$ will contain a $2$-fat $D_3$ minor.
\end{proof}

\end{document}